\documentclass{amsart}
\usepackage{amsfonts}
\usepackage{amsmath,amssymb}
\usepackage{amsthm}
\usepackage{amscd}
\usepackage{graphics}
\usepackage{graphicx}

\theoremstyle{remark}{

\newtheorem{Ex}{{\rm Example}}
\newtheorem{Rem}{{\rm Remark}}

}
\newtheorem{Cor}{Corollary}
\newtheorem{Prop}{Proposition}
\newtheorem{Thm}{Theorem}
\newtheorem{MainThm}{Main Theorem}

\newtheorem{Fact}{Fact}

\begin{document}
\title[
Fold maps on higher dimensional manifolds and Massey products]{Special generic maps and fold maps and information on triple Massey products of higher dimensional differentiable manifolds}
\author{Naoki Kitazawa}
\address{Institute of Mathematics for Industry, Kyushu University, 744 Motooka, Nishi-ku Fukuoka 819-0395, Japan\\
 TEL (Office): +81-92-802-4402 \\
 FAX (Office): +81-92-802-4405 \\
}
\email{n-kitazawa@imi.kyushu-u.ac.jp}
\urladdr{https://naokikitazawa.github.io/NaokiKitazawa.html}
\maketitle
\begin{abstract}

Closed (and simply-connected) manifolds whose dimensions are greater than $4$ are central geometric objects in classical algebraic topology and differential topology. They have been classified via algebraic and abstract objects. On the other hand,
It is difficult to understand them in geometric and constructive ways.

In the present paper, we show such studies via explicit {\it fold} maps, higher dimensional versions of Morse functions, including {\it special generic} maps, generalized versions of Morse functions with exactly two singular points on homotopy spheres. 
The author captured information of the topologies and the differentiable structures of closed (and simply-connected) manifolds which are not so complicated with respect to homotopy and cohomology rings of more general manifolds via construction of these maps previously. In the present paper, as a more precise work, we capture so-called {\it {\rm (}triple{\rm )} Massey products} in this way.

\end{abstract}


\maketitle
\section{Introduction --what will be presented in the present paper, terminologies and notation--.}
\label{sec:1}
Closed (and simply-connected) manifolds whose dimensions are larger than $4$ are central geometric objects in classical algebraic topology and differential topology. They have been classified via algebraic and abstract objects in 1950s--1970s. 
It is difficult to understand them in geometric and constructive ways.  
This paper presents related studies via Morse functions and {\it fold} maps, which are higher dimensional versions of Morse functions, including {\it special generic} maps, generalized versions of Morse functions with exactly two singular points on homotopy spheres.
The author captured information of the topologies and the differentiable structures of these manifolds which are not so complicated with respect to homotopy and cohomology rings of more general manifolds via construction of these maps.
This paper presents a new result on this work by capturing {\it triple Massey products} as more precise information on cohomology classes via construction of fold maps.
\subsection{Notation on differentiable maps and bundles}
\label{subsec:1.1}
Throughout this paper, manifolds and maps between manifolds are fundamental objects and they are smooth and of class $C^{\infty}$. Diffeomorphisms on manifolds are always smooth. The {\it diffeomorphism group} of a manifold is the group of all diffeomorphisms on the manifold. For bundles whose fibers are manifolds, the structure groups are subgroups of the diffeomorphism groups or the bundles are {\it smooth} 
unless otherwise stated. Note that in some scenes, we consider {\rm PL} bundles or bundles whose fibers are polyhedra and structure groups are PL homeomorphisms. 

A {\it linear} bundle is a smooth bundle whose fiber is regarded as a unit sphere or a unit disc in a Euclidean space and whose structure group acts linearly in a canonical way on the fiber.

A {\it singular} point $p \in X$ of a differentiable map $c:X \rightarrow Y$ is a point at which the rank of the differential $dc$ of the map is smaller than the dimension of the target manifold: ${\rm rank} \quad {dc}_p < \dim Y$ holds where ${dc}_p$ denotes the differential of $c$ at $p$. We call the set $S(c)$ of all singular points the {\it singular set} of $c$. We call $c(S(c))$ the {\it singular value set} of $c$. We call $Y-c(S(c))$ the {\it regular value set} of $c$. A {\it singular {\rm (}regular{\rm )} value} is a point in the singular (resp. regular) value set of $c$.

For $x \in {\mathbb{R}}^k$, $||x||$ denotes the distance between $x$ and the origin $0 \in {\mathbb{R}}^k$ where the metric is the Euclidean metric.

\subsection{Fold maps.}
\label{subsec:1.2}
Let $m>n \geq 1$ be integers. A smooth map from an $m$-dimensional smooth manifold with no boundary into an $n$-dimensional smooth manifold with no boundary is said to be a {\it fold} map if at each singular point $p$, the map is represented as
$$(x_1, \cdots, x_m) \mapsto (x_1,\cdots,x_{n-1},\sum_{k=n}^{m-i}{x_k}^2-\sum_{k=m-i+1}^{m}{x_k}^2)$$
 for suitable coordinates and an integer $0 \leq i(p) \leq \frac{m-n+1}{2}$. For singular point $p$, $i(p)$ is unique : it is called the {\it index} of $p$. The set consisting of all singular points of a fixed index of the map is a closed submanifold of dimension $n-1$ with no boundary of the $m$-dimensional manifold. The restriction map to the singular set is an immersion.

\subsection{Special generic maps and what special generic maps and explicit fold maps tell about the topologies and the differentiable structures of the manifolds.}
A {\it special generic} map is a fold map such that the index of each singular point is $0$. A Morse function on a closed manifold with exactly two singular points, characterizing a sphere topologically (except $4$-dimensional cases) as the Reeb's theorem \cite{reeb} states, and the canonical projection of an unit sphere are simplest special generic maps.
 It is an interesting fact that special generic maps restrict the topologies and the differentiable structures of the manifolds admitting them strongly in considerable cases. As an observation for simplest cases, homotopy spheres of dimension $m>3$ do not admit special generic maps into ${\mathbb{R}}^{m-3}$, ${\mathbb{R}}^{m-2}$ and ${\mathbb{R}}^{m-1}$. For integers $m>n \geq 1$, on an $m$-dimensional manifold represented as a connected sum of manifolds represented as products of two standard spheres such that the dimension of either of the two sphere for each manifold is smaller than $n$, we can obtain a special generic map into ${\mathbb{R}}^n$. On the other hand, for example, it is known that $4$-dimensional manifolds homeomorphic to these manifolds and not diffeomorphic to them exist and that they admit fold maps and do not admit special generic maps into ${\mathbb{R}}^3$. For these studies, see \cite{saeki}, \cite{saeki2}, \cite{saekisakuma}, \cite{saekisakuma2} and \cite{wrazidlo} for example. 
Since pioneering studies by Thom (\cite{thom}) and Whitney (\cite{whitney}) on smooth maps on manifolds whose dimensions are larger thanor equal to $2$ into the plane, fold maps have been important tools as the studies on special generic maps show in studying geometric properties of manifolds in the branch of the singularity theory of differentiable maps and application of the theory to geometry of manifolds. These results imply that higher dimensional versions of Morse functions are strong tools in algebraic topology and differential topology of manifolds while Morse functions were well-known to be strong already in 1950s. Motivated by and related to these studies, the author obtained several results.
\begin{Thm}[\cite{kitazawa}, \cite{kitazawa2} and so on.]
\label{thm:1}
Every $7$-dimensional homotopy sphere $M$ admits a fold map $f:M \rightarrow {\mathbb{R}}^4$ satisfying the following properties. 

\begin{enumerate}
\item $f {\mid}_{S(f)}$ is embedding and $f(S(f))=\{x \in {\mathbb{R}}^4 \mid ||x||=1,2,3\}$.
\item The index of each singular point is always $0$ or $1$.
\item For each connected component of the regular value set of $f$, the preimage of a regular value in each connected component is, empty, diffeomorphic to $S^3$, diffeomorphic to $S^3 \sqcup S^3$ and diffeomorphic to $S^3 \sqcup S^3 \sqcup S^3$, respectively.
\end{enumerate}
\end{Thm}
Moreover, the following theorem is obtained. There exist exactly $28$ types of differentiable structures of $7$-dimensional oriented homotopy spheres and exactly $16$ types are obtained as the differentiable structures of total spaces of linear bundles whose fiber are $S^3$ over $S^4$ including that of the standard sphere. All $7$-dimensional homotopy spheres are represented as connected sums of these total spaces.
\begin{Thm}[\cite{kitazawa}, \cite{kitazawa2} and so on.]
\label{thm:2}
A $7$-dimensional homotopy sphere $M$ admits a fold map $f:M \rightarrow {\mathbb{R}}^4$ such that $f {\mid}_{S(f)}$ is embedding and that $f(S(f))=\{x \in {\mathbb{R}}^4 \mid ||x||=1\}$ if and only if $M$ is a standard sphere.
A $7$-dimensional homotopy sphere $M$ admits a fold map $f:M \rightarrow {\mathbb{R}}^4$ such that the following properties hold if and only if $M$ is the total space of a linear bundle whose fiber is $S^3$ over $S^4$.
\begin{enumerate}
\item $f {\mid}_{S(f)}$ is embedding and $f(S(f))=\{x \in {\mathbb{R}}^4 \mid ||x||=1,2\}$.
\item For any connected component $C \subset f(S(f))$ and a small closed tubular neighborhood $N(C)$, the bundle given by the projection given by the composition of $f {\mid}_{f^{-1}(N(C))}$ with the canonical projection to $C$ is trivial.
\item The index of each singular point is always $0$ or $1$.
\item For each connected component of the regular value set of $f$, the preimage of a regular value in each connected component is, empty, diffeomorphic to $S^3$, and diffeomorphic to $S^3 \sqcup S^3$, respectively.
\end{enumerate}
\end{Thm}

We also introduce some of the main results of \cite{kitazawa5}.

\begin{Thm}[\cite{kitazawa5}]
\label{thm:3}
Let $m>n \geq 1$ be integers. Let $k>1$ be an integer satisfying $2k \leq n$, $m-n>k,n-k$, $m-n \neq n$, $k+(m-n) \neq n$ and $(n-k)+(m-n) \neq n$.
Let $\{G_j\}_{j=0}^m$ be a sequence of free and finitely generated commutative groups such that $G_0=G_m=\mathbb{Z}$, that $G_j=G_{m-j}$ for $0 \leq j \leq m$ and that $G_j$ is zero except the case $j=0,k,n-k,n,m-n,m-n+k,m-k,m$. 
In this situation, there exist a closed and simply-connected manifold $M$ of dimension $m$ such that homology group is free and that $H_j(M;\mathbb{Z})$ is isomorphic to $G_j$ and a fold map $f:M \rightarrow {\mathbb{R}}^n$ such that $f {\mid}_{S(f)}$ is embedding, that the index of each singular point is always $0$ or $1$, and that for each connected component of the regular value set of $f$, the preimage of a regular value in each connected component is, empty, diffeomorphic to $S^3$, or diffeomorphic to $S^3 \sqcup S^3$.
Furthermore, if $G_k$, $G_{n-k}$ and $G_n$ are non-trivial groups, then there exist infinitely many closed and simply-connected manifolds $M_{\lambda}$ of dimension $m$ such that $H_j(M_{\lambda};\mathbb{Z})$ is isomorphic to $G_j$ and that the integral cohomology rings of $M_{{\lambda}_1}$ and $M_{{\lambda}_2}$ are not isomorphic for distinct ${\lambda}_1,{\lambda}_2 \in \Lambda$ admitting the fold maps as before.
\end{Thm}

For example, we can set $(m,n,k)=(7,4,2)$.
Through this, for example, we can see that the topologies of singular value sets and the preimages affect cohomology rings of the manifolds. In short, we can catch information of the cohomology rings of manifolds via construction of explicit fold maps. See also \cite{kitazawa4} as a related work.

We obtained these results in trying to obtain information of the topologies and the differentiable structures of fundamental closed and simply-connected manifolds such as homotopy spheres and manifolds represented as their products and connected sums of  these fundamental manifolds and cohomology rings of more general closed and simply-connected manifolds via explicit fold maps.
\subsection{Main theorems: triple Massey products of closed and simply-connected manifolds admitting special generic maps.}
The main theorems of the present paper are the following two. We present some statements of the first main theorem leaving the terminologies and notation we need lalter. This captures much information of cohomology rings and {\it triple Massey products} of closed and simply-connected manifolds of suitable classes via explicit special generic maps on them.
\begin{MainThm}[A partial explanation of Theorem \ref
{thm:4} and Theorem \ref{thm:5}]
\label{mainthm:1}
Let $A$ be a PID.
Take an arbitrary compact and simply-connected polyhedron $X_i$ of dimension $i \geq 5$ having free homology modules and triplets of three integral cohomology classes of degree $2$ for which we can define the
 triple Massey products which do not vanish such that the following properties hold as ones in \cite{dranishnikovrudyak} or \cite{dranishnikovrudyak2}.
\begin{enumerate}
\item $H_i(X_{i};A)$ is not zero.
\item For any pair $c_1 \in H^{j_1}(X_{i};A)$ and $c_2 \in H^{j_2}(X_{i};A)$ of cohomology classes satisfying $2 \leq j_1, j_2 \leq i-1$, the cup product of $c_1$ and $c_2$ always vanishes.
\end{enumerate}

In this situation, for arbitrary integers $m$ and $n$ satisfying $n \geq 2i+1$ and $m-i-1 \geq i$, there exist a closed and simply-connected manifold $M_{i,m,n}$ having a triplet of three integral cohomology classes of degree $2$ for which we can define the
 triple Massey product which does not vanish and a special generic map $f_{i,m,n}:M_{i,m,n} \rightarrow {\mathbb{R}}^n$ satisfying the following properties.
\begin{enumerate}
\item The Reeb space $W_{f_{i,m,n}}$ collapses to $X_i$.
\item The restriction of the map $f_{i,m,n}$ to the singular set is an embedding.
\item The $\leq m-i-1$-algebra of $H^{\ast}(M_{i,m,n};A)$ is isomorphic to $H^{\ast}(X_i;A)$.
\item For any pair of cohomology classes $C_{c_1} \in H^{j_1}(M;A)$ and $C_{c_2} \in H^{j_2}(M;A)$ satisfying $0 \leq j_1,j_2 \leq m-i-1$ and  $j_1+j_2 \geq i+1$, the product is zero.
\item For any triplet $(C_{c_1},C_{c_2},C_{c_3}) \in H^{j_1}(W_f;A) \times H^{j_2}(W_f;A) \times H^{j_3}(W_f;A)$ of cohomology classes the triple Massey product $<C_{c_1},C_{c_2},C_{c_3}>$ of which we can define, we can define the triple Massey product $<{q_{f_{i,m,n}}}^{\ast}(C_{c_1}),{q_{f_{i,m,n}}}^{\ast}(C_{c_2}),{q_{f_{i,m,n}}}^{\ast}(C_{c_3})>$ of the triplet $({q_{f_{i,m,n}}}^{\ast}(C_{c_1}),{q_{f_{i,m,n}}}^{\ast}(C_{c_2}),{q_{f_{i,m,n}}}^{\ast}(C_{c_3})) \in H^{j_1}(M;A) \times H^{j_2}(M;A) \times H^{j_3}(M;A)$ with $0 \leq j_1,j_2,j_3 \leq m-i-1$ and $j_1+j_2+j_3 \geq i+2$ and this always vanishes.
\end{enumerate}
\end{MainThm}

\begin{MainThm}[Theorem \ref{thm:6} and Theorem \ref{thm:7}]
\label{mainthm:2}
Let $A$ be a PID.
\begin{enumerate}
\item A $7$-dimensional closed and simply-connected manifold $M$ having a triplet of three cohomology classes of degree $2$ for which we can define the
 triple Massey product which does not vanish under the assumption that the coefficient ring is $A$ never admits a special generic map $f:M \rightarrow {\mathbb{R}}^n$ for $n=1,2,3,4,5$.
\item  Let $n=4,5$ and let $\{G_j\}_{j=0}^7$ be a sequence of free and finitely generated modules over $A$ such that $G_0=G_m=A$, that $G_j=G_{7-j}$ for $0 \leq j \leq 7$, that $G_1=G_6=\{0\}$ and that the rank of $G_2$ is greater than or equal to $4$.
In this situation, there exists a pair $(M_1,M_2)$ of $7$-dimensional closed and simply-connected manifolds satisfying the following properties.
\begin{enumerate}
\item $H_j(M_i;A) \cong G_j$ for $i=1,2$ and $0 \leq j \leq 7$.
\item The cohomology rings $H^{\ast}(M_1;A)$ and $H^{\ast}(M_2;A)$ are isomorphic. 
\item $M_1$ admits a special generic map into ${\mathbb{R}}^n$.
\item $M_2$ admits no special generic map into ${\mathbb{R}}^n$.
\item $M_2$ admits a fold map into ${\mathbb{R}}^n$  
\end{enumerate}
This holds for $n=3$ if we add the constraint $G_3=\{0\}$.
\end{enumerate}
\end{MainThm}

This new result is a kind of theorems on restrictions on the topologies and the differentiable structures of manifolds admitting special generic maps into the Euclidean space of a fixed dimension, discussed in \cite{saeki}, \cite{saeki2}, \cite{saekisakuma}, \cite{saekisakuma2}, \cite{wrazidlo} and so on. Corollary \ref{cor:1} is
 also a kind of this in the present paper, which was already shown in \cite{kitazawa5}.

Last, to know famous facts on $7$-dimensional closed and simply-connected manifolds and that the class is attractive, see \cite{eellskuiper}, \cite{milnor} and so on as great classical well-known works and \cite{kreck} and so on as recent interesting works via concrete bordism theory and related theory.

\subsection{The content of the present paper}
\label{subsec:1.4}

The organization of the paper is as the following.
In the next section, we review the {\it Reeb space} of a fold map. The {\it Reeb space} of a fold map is a polyhedron whose dimension is equal to the dimension of the target space, defined as the space of all connected components of preimages of the map. This object inherits topological information of the manifold such as homology groups, cohomology rings and so on much in several situations. We review an explicit situation explaining this well and important in the present paper as Proposition \ref{prop:1}. The class of maps we can apply this contains the class of special generic maps as a proper subclass and maps in the previous subsection.  
The third section is for explicit special generic maps. The last section is devoted to main theorems. We utilize several algebraic topological and differential topological theory on manifolds and polyhedra. As an important tool, we explain several notions on cohomology classes such as {\it {\rm (}triple{\rm )} Massey products} without precise explanations. This together with theory in the second and third sections yields main theorems.

\section{Reeb spaces of fold maps and homology groups and cohomology rings of the manifolds.}
\label{sec:2}

For a continuous map $c:X \rightarrow Y$ between topological spaces, we can define an equivalence relation ${\sim}_c$ on $X$ by the following rule: $p_1 {\sim}_c p_2$ if and only if $p_1$ and $p_2$ are in a same connected component of a preimage $c^{-1}(q)$ ($q \in Y$). We call the quotient space $W_c:=X/{{\sim}_c}$ the {\it Reeb space} of $c$ and denote the quotient map by $q_c:X \rightarrow W_c$. We can define in a unique way and also define a map $\bar{f}$ satisfying $f=\bar{f} \circ q_f$. The Reeb space is a $\dim Y$-dimensional polyhedron for a smooth map of a considerable wide class such that the dimension of the domain is greater than that of the target space. Fold maps satisfy
 the property. \cite{kobayashisaeki}, \cite{shiota} and so on show this.

A {\it simple} fold map $f:M \rightarrow N$ from a closed, connected and orientable manifold of dimension $m$ into an $n$-dimensional manifold $N$ with no boundary satisfying $m>n$ is a fold map such that $q_f {\mid}_{S(f)}$ is injective. 
For any connected component $C$ of $q_f(S(f))$, which is regarded as a closed smooth manifold, let $N(C)$ be a small regular neighborhood of $C$. This is regarded as a PL bundle over $C$ whose fiber is a closed interval or a Y-shaped $1$-dimensional polyhedron. The composition of $q_{f} {\mid}_{{q_{f}}^{-1}(N(C))}$ with a canonical projection to $C$ gives a smooth trivial bundle. 

The following proposition states that the Reeb space of a fold map of a suitable class inherits some fundamental algebraic topological information of the manifold. A {\rm PID} means a so-called {\it principal ideal domain} having a unique identity element $1$ different from the zero element.  

\begin{Prop}[\cite{kitazawa2}, \cite{kitazawa3}, \cite{kitazawa4}, \cite{saekisuzuoka} and so on.]
\label{prop:1}
Let $A$ be a commutative group. 
Let $M$ be a closed, connected and orientable manifold of dimension $m$, $N$ be an $n$-dimensional manifold with no boundary and $f:M \rightarrow N$ be a simple fold map such that preimages of regular values
 are always disjoint unions of copies of $S^{m-n}$ and that indices of singular points are always $0$ or $1$. Thus two induced homomorphisms ${q_f}_{\ast}:{\pi}_j(M) \rightarrow {\pi}_j(W_f)$, ${q_f}_{\ast}:H_j(M;A) \rightarrow H_j(W_f;A)$, and ${q_f}^{\ast}:H^j(W_f;A) \rightarrow H^j(M;A)$ are isomorphisms of groups and modules for $0 \leq j \leq m-n-1$. 

Furthermore, we have the following properties for specific cases.
\begin{enumerate}
\item Let $A$ be a commutative ring. Let $J$ be the set of integers smaller than or equal to $0$ and larger than or equal to $m-n-1$ and let ${\oplus}_{j \in J} H^{j}(W_f;A)$ and ${\oplus}_{j \in J} H^{j}(M;A)$ be the algebras obtained by replacing the $j$-th modules of the cohomology rings $H^{\ast}(W_f;A)$ and $H^{\ast}(M;A)$ by $\{0\}$ for $j \geq m-n$, respectively. In this situation, $q_f$ induces an isomorphism between the commutative algebras
${\oplus}_{j \in J} H^j(W_f;A)$ and ${\oplus}_{j \in J} H^{j}(M;A)$, defined as the restriction of the original homomorphism ${q_f}^{\ast}$. 
\item Furthermore, if $A$ is a PID and $m=2n$ holds, then the rank of $M$ is twice the rank of $W_f$ and in addition if $H_{n-1}(W_f;A)$, which is isomorphic to $H_{n-1}(M;A)$, is free, then these two modules are also free.
\end{enumerate}
Special generic maps satisfy the assumption and we can replace the inequality $0 \leq j \leq m-n-1$ by $0 \leq j \leq m-n$. 
\end{Prop}

We present additional fundamental and important facts related to this. 

The following proposition is important in the next section.

\begin{Prop}[\cite{kitazawa2}, \cite{kitazawa3} and so on.]
\label{prop:2}
In Proposition \ref{prop:1}, $M$ is in the {\rm PL} category bounds a compact, connected and orientable manifold $W$ collapsing to $W_f$. $q_f$ is represented as the composition of the canonical inclusion $i:M \rightarrow W$ and the map representing the collapsing. Moreover, if $f$ is special generic, then $W_f$ collapses to an ($n-1$)-dimensional polyhedron.
\end{Prop}

The following two propositions will be useful in constructing examples for the main theorems later.

\begin{Prop}
\label{prop:3}
In the situation of Proposition \ref{prop:1}, there exists a simple fold map $f_0:M_0 \rightarrow N$ such that preimages of regular values
 are always disjoint unions of copies of $S^{m-n}$ and that indices of singular points are always $0$ or $1$ on a closed, connected and orientable manifold $M_0$ of dimension $m$ satisfying the following properties.
\begin{enumerate}
\item $W_{f_0}=W_f$.
\item $\bar{f_0}=\bar{f}$.
\item At least one, two or three of the following properties hold.
\begin{enumerate}
\item For any connected component $C$ of $q_f(S(f))=q_{f_0}(S(f_0))$, which is regarded as a closed smooth manifold, let $N(C)$ be a small regular neighborhood of $C$, regarded as a PL bundle over $C$ whose fiber is a closed interval or a Y-shaped $1$-dimensional polyhedron, the composition of $q_{f_0} {\mid}_{{q_{f_0}}^{-1}(N(C))}$ with a canonical projection to $C$ gives a smooth trivial bundle.  
\item For any connected component $R$ of $W_f-{\bigcup}_{C} N(C)$, which is regarded as a compact
 smooth manifold, $q_f {\mid}_{{q_f}^{-1}(R)}$ gives a smooth trivial bundle over $R$ whose fiber is $S^{m-n}$.  
\item $W$ in Proposition \ref{prop:2} can be taken as a smooth manifold.
\end{enumerate}
\end{enumerate}  
\end{Prop}

\begin{Prop}[\cite{saeki} and so on]
\label{prop:4}
\begin{enumerate}
\item The Reeb space of a special generic map $f:M \rightarrow N$ from a closed and connected manifold $M$ of dimension $m$ into an $n$-dimensional manifold $N$ with no boundary satisfying $m>n$ is a smooth manifold of dimension $n$ immersed into $N$ via $\bar{f}:W_f \rightarrow N$.
\item For any smooth immersion ${\bar{f}}_N$ of a compact and connected manifold $\bar{N}$ of dimension $n>0$ into an $n$-dimensional manifold $N$ with no boundary and any integer $m>n$, there exists a closed and connected manifold $M$ of dimension $m$ and a special generic map $f:M \rightarrow N$ satisfying $W_f=\bar{N}$ and $\bar{f}={\bar{f}}_N$. If $N$ is orientable, then $M$ can be taken as an orientable manifold.
\end{enumerate}
\end{Prop}

\section{Explicit special generic maps.}
\label{sec:3}
The following example explains a family of simplest explicit special generic maps.
\begin{Ex}
\label{ex:1}
Let $M$ be a closed manifold of dimension $m>1$ represented as a connected sum of all manifolds in the family $\{S^{k_j} \times S^{m-k_j}\}$ of finitely many manifolds satisfying $1 \leq k_j <n$ where $1<n<m$ holds. This admits a special generic map into ${\mathbb{R}}^n$ such that the restriction to the singular set is an embedding and that the image is represented as a boundary connected sum of all manifolds in the family $\{S^{k_j} \times D^{n-k_j}\}$. 
\end{Ex}

Proposition \ref{prop:5} later yields the following corollary.
\begin{Cor}
\label{cor:1}
Let $A$ be a PID. Let $M$ be a closed manifold of dimension $m>1$. If it admits a special generic map into ${\mathbb{R}}^n$ satisfying $m-n \geq n-1$, then the product of two cohomology classes $c_1 \in H^{i_1}(M;A)$ and $c_2 \in H^{i_2}(M;A)$ must vanish for $1 \leq i_1,i_2 \leq m-n$ with $n \leq i_1+i_2$.
\end{Cor}

We can apply this by setting $(m,n)=(7,3),(7,4)$ to know that the product of two cohomology classes $c_1 \in H^{i_1}(M;A)$ and $c_2 \in H^{i_2}(M;A)$ must vanish for $1 \leq i_1,i_2 \leq m-n$ with $n \leq i_1+i_2$ for example posing an additional condition that the $7$-dimensional manifolds are simply-connected. In both cases, for $2 \leq i_1,i_2 \leq 3$ satisfying $4 \leq i_1+i_2$, we can see this.

Theorem \ref{thm:3} presents infinite families of manifolds admitting no special generic maps into ${\mathbb{R}}^n$: Corollary \ref{cor:1} was already shown in \cite{kitazawa5} essentially and applied there.
\section{The main theorems.} 
For a commutative ring or graded commutative algebra $A$ and its subset $S_A$, $<S_A>$ denotes the subalgebra generated by the set $S_A$. 
For a topological space, $\delta$ denotes the coboundary operator on the cochain complex and $\cup$ denotes the cup product.  
For a cocycle $c$, we denote $C_c$ the cohomology class represented by this.

For a graded commutative algebra $A$, we define the {\it $i$-th module} as the module consisting of all elements of degree $i \geq 0$ of $A$. For an integer $i \geq 0$, we set the {\it $\leq i$-algebra} $A_{\leq i}$ of $A$ as the following graded commutative algebra.
\begin{enumerate}
\item As a graded commutative module, $A_{\leq i}$ is a module obtained by replacing the $j$-th module of $A$ by the trivial module for $j>i$ without changing the $j$-th module for $j \leq i$.   
\item As a graded commutative algebra, the product of a pair of elements in $A_{\leq i}$ is same as that in $A$ for two elements whose degrees are $j_1 \geq 0$ and $j_2 \geq 0$, respectively, unless $j_1+j_2 > i$.     
\end{enumerate}   

Via this notion, we can give a simpler explaination of isomorphisms of commutative algebras of Proposition \ref{prop:1} for example.

Let $A$ be a PID. For a topological space $X$ and three cohomology classes $C_{c_i} \in H^{\ast}(X;A)$ (i=1,2,3) representing a cocycle $c_i$ such that the cup product of $C_{c_1}$ and $C_{c_2}$ and that of $C_{c_2}$ and $C_{c_3}$ vanish, we can define
 an element represented by a cocycle of the form $u_1 {c_3}^{\prime}+(-1)^{{\rm deg} {c_1}^{\prime}+1}{c_1}^{\prime} u_2$ with ${\delta}(u_1)={c_1}^{\prime} \cup {c_2}^{\prime} \in C_{c_1} \cup C_{c_2}=0 \in H^{\ast}(X;A)$, ${\delta}(u_2)={c_2}^{\prime} \cup {c_3}^{\prime} \in C_{c_2} \cup C_{c_3}=0 \in H^{\ast}(X;A)$ and ${c_j}^{\prime} \in C_{c_j}$ of the quotient module $H^{({\Sigma}_{j=1}^3 {\rm deg} c_j)-1}(X;A)/<\{C_{c_1}{C_{c_2,c_3}}^{\prime}+{C_{c_1,c_2}}^{\prime}C_{c_3} \mid {C_{c_2,c_3}}^{\prime} \in H^{{\rm deg} c_2 + {\rm deg} c_3 - 1}(X;A), {C_{c_1,c_2}}^{\prime} \in H^{{\rm deg} c_1+ {\rm deg} c_2 - 1}(X;A)\}>$.  

\begin{Fact}
We can define the element in a unique way.
\end{Fact}

We call the element the {\it triple Massey product} of the triplet of the three cohomology classes and denote it by $<C_{c_1},C_{c_2},C_{c_3}> \in H^{({\Sigma}_{j=1}^3 {\rm deg} c_j)-1}(X;A)/<\{C_{c_1}{C_{c_2,c_3}}^{\prime}+{C_{c_1,c_2}}^{\prime}C_{c_3} \mid {C_{c_2,c_3}}^{\prime} \in H^{{\rm deg} c_2 + {\rm deg} c_3 - 1}(X;A), {C_{c_1,c_2}}^{\prime} \in H^{{\rm deg} c_1+ {\rm deg} c_2 - 1}(X;A)\}>$. 
If it is zero, then we say that the triple Massey product {\it vanishes}.
For this and related topics, see \cite{kraines}, \cite{masseyuehara}, \cite{taylor} and \cite{taylor2} for example. We will apply
 fundamental propositions and theorems on triple Massey products which are explained precisely there without precise explanations in this section.

\begin{Prop}
\label{prop:5}
\begin{enumerate}
\item In the situation of Proposition \ref{prop:1}, if $W_f$ collapses to an $i$-dimensional polyhedron satisfying $0 \leq i \leq n$, then we can replace"$m-n-1$" by "$m-i-1$".
\item In the situation of Proposition \ref{prop:1} where $A$ is a PID, for two cohomology classes $C_{c_1} \in H^{j_1}(M;A)$ and $C_{c_2} \in H^{j_2}(M;A)$ satisfying $0 \leq j_1,j_2 \leq m-n-1$ and  $j_1+j_2 \geq n+1$, the product is zero. Furthermore, if the map $f$ is special generic, then we can replace the inequalities by $0 \leq j_1,j_2 \leq m-n$ and $j_1+j_2 \geq n$. More generally, if $W_f$ collapses to an $i$-dimensional polyhedron satisfying $0 \leq i \leq n$, then we can replace the inequalities by $0 \leq j_1,j_2 \leq m-i-1$ and $j_1+j_2 \geq i+1$.
\item In the situation of Proposition \ref{prop:1} where $A$ is a PID, for any triplet $(C_{c_1},C_{c_2},C_{c_3}) \in H^{j_1}(W_f;A) \times H^{j_2}(W_f;A) \times H^{j_3}(W_f;A)$ of cohomology classes the triple Massey product 
$<C_{c_1},C_{c_2},C_{c_3}>$ of which we can define, we can define the triple Massey product $<{q_f}^{\ast}(C_{c_1}),{q_f}^{\ast}(C_{c_2}),{q_f}^{\ast}(C_{c_3})>$ of the triplet $({q_f}^{\ast}(C_{c_1}),{q_f}^{\ast}(C_{c_2}),{q_f}^{\ast}(C_{c_3})) \in H^{j_1}(M;A) \times H^{j_2}(M;A) \times H^{j_3}(M;A)$ with $0 \leq j_1,j_2,j_3 \leq m-n-1$. If the former element $<C_{c_1},C_{c_2},C_{c_3}>$ vanishes {\rm (}does not vanish{\rm )} and the inequality $0 \leq j_1+j_2+j_3-1 \leq m-n-1$ holds, then the latter element $<{q_f}^{\ast}(C_{c_1}),{q_f}^{\ast}(C_{c_2}),{q_f}^{\ast}(C_{c_3})>$ vanishes {\rm (}resp. does not vanish{\rm )}. Furthermore, if the map $f$ is special generic, then we can replace "$m-n-1$" by "$m-n$". More generally, if $W_f$ collapses to an $i$-dimensional polyhedron satisfying $0 \leq i \leq n$, then we can replace "$m-n-1$" by "$m-i-1$".
\item In the situation of Proposition \ref{prop:1} where $A$ is a PID, for any triplet $(C_{c_1},C_{c_2},C_{c_3}) \in H^{j_1}(W_f;A) \times H^{j_2}(W_f;A) \times H^{j_3}(W_f;A)$ of cohomology classes the triple Massey product $<C_{c_1},C_{c_2},C_{c_3}>$ of which we can define, we can define the triple Massey product $<{q_f}^{\ast}(C_{c_1}),{q_f}^{\ast}(C_{c_2}),{q_f}^{\ast}(C_{c_3})>$ of the triplet $({q_f}^{\ast}(C_{c_1}),{q_f}^{\ast}(C_{c_2}),_{q_f}^{\ast}(C_{c_3})) \in H^{j_1}(M;A) \times H^{j_2}(M;A) \times H^{j_3}(M;A)$ with $0 \leq j_1,j_2,j_3 \leq m-n-1$ and $j_1+j_2+j_3 \geq n+2$ and this vanishes. Furthermore, if the map $f$ is special generic, then we can replace "$m-n-1$" by "$m-n$" and "$n+2$" by "$n+1$". More generally, if $W_f$ collapses to an $i$-dimensional polyhedron satisfying $0 \leq i \leq n$, then we can replace "$m-n-1$" by "$m-i-1$" and "$n+2$" by "$i+2$".
\end{enumerate}
\end{Prop}
\begin{proof}

In the situation of Proposition \ref{prop:2}, $W$ is obtained by attaching handles to $M \times \{1\} \subset M \times [0,1]$ whose indices are larger than or equal to $m-n+1$. We can replace "$m-n+1$" by "$m-n+2$" if $f$ is special generic and by "$m+1-i$" if $W_f$ collapses to an $i$-dimensional polyhedron satisfying $0 \leq i \leq n$. This is a key ingredient: see also \cite{kitazawa3} and \cite{saekisuzuoka} for example. The inclusion $i:M \rightarrow W$ induces an isomorphism $i^{\ast}:H^j(W;A) \rightarrow H^j(M;A)$ for $0 \leq j \leq m-n-1$: we can replace "$m-n-1$" by "$m-n$" if $f$ is special generic and by "$m+1-i$" if $W_f$ collapses to an $i$-dimensional polyhedron satisfying $0 \leq i \leq n$. 
We denote the map giving collapsing to $W_f$ or a lower dimensional polyhedron of $W$ in Proposition \ref{prop:2} by $d$. We immediately have the first fact. 
The second fact was shown in \cite{kitazawa5}. $C_{c_1} \cup C_{c_2}=i^{\ast}({i^{\ast}}^{-1}(C_{c_1}) \cup {i^{\ast}}^{-1}(C_{c_2}))=0$ by the simple homotopy type of $W_f$.
The third fact follows by Propositions \ref{prop:1} and \ref{prop:2} together with topological properties of the maps $i$ and $d$ immediately.
The fourth fact also follows by similar reasons and we will explain this. $<{i^{\ast}}^{-1}({q_f}^{\ast}(C_{c_1})),{i^{\ast}}^{-1}({q_f}^{\ast}(C_{c_2})),{i^{\ast}}^{-1}({q_f}^{\ast}(C_{c_3}))>=<d^{\ast}(C_{c_1}),d^{\ast}(C_{c_2}),d^{\ast}(C_{c_3})$ is defined and it vanishes. 
We discuss properties of related cocycles on an $i$-dimensional polyhedron the polyhedron $W_f$ collapsed to and consider the pullbacks to $W_f$ and the ($m+1$)-dimensional manifold $W$ and we can see that the vanishing class is obtained by constructing the zero cocycle. By a fundamental property of the triple Massey products, we can see that $<{q_f}^{\ast}(C_{c_1}),{q_f}^{\ast}(C_{c_2}),{q_f}^{\ast}(C_{c_3})>$ vanishes.
\end{proof}

Note also for example that by Proposition \ref{prop:5}, we can know the product of two cohomology classes of $c_1 \in H^{j_1}(M;A)$ and $c_2 \in H^{j_2}(M;A)$ except cases where $j_1 \geq m-i$, $j_2 \geq m-i$ or $m-i \leq j_1+j_2 \leq i$ holds via the Reeb space. As another related remark, \cite{kitazawa5} show Theorem \ref{thm:3} with concrete and complete descriptions of the (integral) cohomology rings of the manifolds of some classes via compact submanifolds with no boundaries representing corresponding homology classes of the manifolds and their Poincar\'e duals. 

\begin{Thm}
\label{thm:4}
Let $A$ be a PID.
For any integer $i \geq 5$, there exist infinitely many compact and simply-connected polyhedra $X_{i,k}$ of dimension $i$ whose homology modules are free and which have triplets of three integral cohomology classes of degree $2$ for which we can define the
 triple Massey products which do not vanish such that the following properties hold.
\begin{enumerate}
\item $H_i(X_{i,k};A)$ is not zero.
\item For distinct $k_1$ and $k_2$, the modules $H_{\ast}(X_{i,k_1};A)$ and $H_{\ast}(X_{i,k_2};A)$ are not isomorphic.
\item For any pair $c_1 \in H^{j_1}(X_{i,k};A)$ and $c_2 \in H^{j_2}(X_{i,k};A)$ of cohomology classes satisfying $2 \leq j_1, j_2 \leq i-1$, the cup product of $c_1$ and $c_2$ always vanishes.
\end{enumerate}

Take an arbitrary compact and simply-connected polyhedron $X_i$ of dimension $i$ having triplets of three integral cohomology classes of degree $2$ for which we can define the
 triple Massey products which do not vanish such that the following properties hold.
\begin{enumerate}
\item $H_i(X_{i};A)$ is not zero.
\item For any pair $c_1 \in H^{j_1}(X_{i};A)$ and $c_2 \in H^{j_2}(X_{i};A)$ of cohomology classes satisfying $2 \leq j_1, j_2 \leq i-1$, the cup product of $c_1$ and $c_2$ always vanishes.
\end{enumerate}

In this situation, for arbitrary integers $m$ and $n$ satisfying $m>n \geq 2i+1$ and $m-i-1 \geq i$, there exist a closed and simply-connected manifold $M_{i,m,n}$ having a triplet of three integral cohomology classes of degree $2$ for which we can define the
 triple Massey product which does not vanish and a special generic map $f_{i,m,n}:M_{i,m,n} \rightarrow {\mathbb{R}}^n$ satisfying the following properties.
\begin{enumerate}
\item The Reeb space $W_{f_{i,m,n}}$ collapses to $X_i$.
\item The restriction of the map $f_{i,m,n}$ to the singular set is an embedding.
\item The $\leq m-i-1$-algebra of $H^{\ast}(M_{i,m,n};A)$ is isomorphic to $H^{\ast}(X_i;A)$.
\item For any pair of cohomology classes $C_{c_1} \in H^{j_1}(M;A)$ and $C_{c_2} \in H^{j_2}(M;A)$ satisfying $0 \leq j_1,j_2 \leq m-i-1$ and  $j_1+j_2 \geq i+1$, the product is zero.
\item For any triplet $(C_{c_1},C_{c_2},C_{c_3}) \in H^{j_1}(W_f;A) \times H^{j_2}(W_f;A) \times H^{j_3}(W_f;A)$ of cohomology classes the triple Massey product $<C_{c_1},C_{c_2},C_{c_3}>$ of which we can define, we can define the triple Massey product $<{q_{f_{i,m,n}}}^{\ast}(C_{c_1}),{q_{f_{i,m,n}}}^{\ast}(C_{c_2}),{q_{f_{i,m,n}}}^{\ast}(C_{c_3})>$ of the triplet $({q_{f_{i,m,n}}}^{\ast}(C_{c_1}),{q_{f_{i,m,n}}}^{\ast}(C_{c_2}),{q_{f_{i,m,n}}}^{\ast}(C_{c_3})) \in H^{j_1}(M;A) \times H^{j_2}(M;A) \times H^{j_3}(M;A)$ with $0 \leq j_1,j_2,j_3 \leq m-i-1$ and $j_1+j_2+j_3 \geq i+2$ and this always vanishes.
\end{enumerate}
\end{Thm}
\begin{proof}
We can easily see that for any integer $i \geq 5$, there exist infinitely many compact and simply-connected polyhedra $X_{i,k}$ of dimension $i$ having triplets of three integral cohomology classes of degree $2$ for which we can define the
 triple Massey products which do not vanish satisfying the first three listed properties. For example, we take a $5$-dimensional compact and simply-connected polyhedron in \cite{dranishnikovrudyak} or \cite{dranishnikovrudyak2} or a bouquet of the polyhedron and an arbitrary finite number of copies of standard spheres whose dimensions are greater than $1$. We can embed this into ${\mathbb{R}}^n$ to obtain a regular neighborhood and by Proposition \ref{prop:4} we have a closed and simply-connected manifold $M_{i,m,n}$ and a special generic map $f_{i,m,n}:M_{i,m,n} \rightarrow {\mathbb{R}}^n$ such that the restriction to the singular set is an embedding and that the Reeb space $W_{f_{i,m,n}}$ is diffeomorphic to the regular neighborhood and collapses to the original polyhedron before. Proposition \ref{prop:5} completes the proof.
\end{proof}

\begin{Rem}
\label{rem:1}
A $5$-dimensional compact and simply-connected polyhedron in \cite{dranishnikovrudyak} or \cite{dranishnikovrudyak2} can be embedded into ${\mathbb{R}}^8$ in fact.  In this case, we can replace the condition $n \geq 2i+1=2\times 5+1=11$ by $n \geq 8$ in Theorem \ref{thm:4}. 
We can see that for the $5$-dimensional compact and simply-connected polyhedron $X_{5,0}$, we can see that $H_j(X_{5,0};A)$ vanishes for $j \neq 0,2,5$, that $H_2(X_{5,0};A)$ is free and of rank $3$, and that $H_j(X_{5,0};A)$ is isomorphic to $A$ for $j=0,5$.
\end{Rem}
\begin{Rem}
\label{rem:2}
In Theorem \ref{thm:4}, we can take an $i$-dimensional compact and simply-connected polyhedron as a bouquet of standard spheres whose dimensions are greater than $1$ instead of $X_i$, whose homology module and cohomology ring are isomorphic to those of the original polyhedron. We can obtain a closed and simply-connected manifold $M_{i,m,n,0}$ having no triplet of three integral cohomology classes of degree $2$ for which we can define the
 triple Massey product which does not vanish and a special generic map $f_{i,m,n,0}:M_{i,m,n,0} \rightarrow {\mathbb{R}}^n$ satisfying the five similar properties.
\end{Rem}

\begin{Thm}
\label{thm:5}
Let $A$ be a PID.
For any integer $i \geq 7$, there exist infinitely many compact and simply-connected polyhedra $X_{i,k}$ of
 dimension $i$ which have triplets of three integral cohomology classes of degree $2$ for which we can define the
 triple Massey products which do not vanish such that the following properties hold.
\begin{enumerate}
\item $H_i(X_{i,k};A)$ is not zero.
\item For distinct $k_1$ and $k_2$, the modules $H_{\ast}(X_{i,k_1};A)$ and $H_{\ast}(X_{i,k_2};A)$ are not isomorphic.
\item $X_i$ is represented as a bouquet of polyhedra which are PL, closed and connected manifolds.
\end{enumerate}

Take an arbitrary compact and simply-connected polyhedron $X_i$ of dimension $i$ having triplets of three integral cohomology classes of degree $2$ for which we can define the
 triple Massey products which do not vanish such that the following properties hold.
\begin{enumerate}
\item $H_i(X_{i};A)$ is not zero.
\item $X_i$ is represented as a bouquet of polyhedra which are PL, closed and simply-connected manifolds.
\end{enumerate}

In this situation, for arbitrary integers $m$ and $n$ satisfying $m>n \geq 2i$ and $m-i-1 \geq i$, there exist a closed and simply-connected manifold $M_{i,m,n}$ having a triplet of three integral cohomology classes of degree $2$ for which we can define the
 triple Massey product which does not vanish and a special generic map $f_{i,m,n}:M_{i,m,n} \rightarrow {\mathbb{R}}^n$ satisfying the five properties in Theorem \ref{thm:4}.
\end{Thm}
\begin{proof}

We can easily see that for any integer $i \geq 7$, there exist infinitely many compact and simply-connected polyhedra $X_{i,k}$ of dimension $i$ having triplets of three integral cohomology classes of degree $2$ for which we can define the
 triple Massey products which do not vanish satisfying the first three listed properties. For example, we take an $i$-dimensional closed and simply-connected
  manifold in \cite{dranishnikovrudyak}, \cite{dranishnikovrudyak2} or \cite{fernandezMunoz} or a bouquet of the manifold and an arbitrary finite number of copies of standard spheres whose dimensions are greater than $1$. We can show the present theorem as Theorem \ref{thm:4}. 
\end{proof}
As another main theorem, we immediately obtain the following theorem by applying Proposition \ref{prop:5} and so on.

\begin{Thm}
\label{thm:6}
Let $A$ be a PID.
A $7$-dimensional closed and simply-connected manifold $M$ having a triplet of three cohomology classes of degree $2$ for which we can define the
 triple Massey product which does not vanish under the assumption that the coefficient ring is $A$ never admits a special generic map $f:M \rightarrow {\mathbb{R}}^n$ for $n=1,2,3,4,5$.
\end{Thm}

We give some explanations on the polyhedron $X_{5,0}$ and a regular neighborhood $N(X_{5,0})$, which is a compact and simply-connected $8$-dimensional manifold obtained by embedding this into ${\mathbb{R}}^8$, in \cite{dranishnikovrudyak}, \cite{dranishnikovrudyak2} and Remark \ref{rem:1}. $\partial N(X_{5,0})$ is a $7$-dimensional closed and simply-connected manifold. Let $A$ be an arbitrary PID. We have the following exact sequence with some isomorphisms of modules

\ \\

$\begin{CD}
@>   >> H^j(N(X_{5,0});A) \cong \{0\}\\
@> >> H^j(\partial N(X_{5,0});A) \\ @> >> H^{j+1}(N(X_{5,0}),\partial N(X_{5,0});A) \cong H_{7-j}(N(X_{5,0}),A) \cong \{0\} @> >>
\end{CD}$

\ \\

for $j=3,4$, the following exact sequence with some isomorphisms of modules
 
\ \\

$$\begin{CD}
@> >>H^2(N(X_{5,0}),\partial N(X_{5,0});A) \cong H_{6}(N(X_{5,0}),A) \cong \{0\} \\
@> >> H^2(N(X_{5,0});A) \cong A \oplus A \oplus A\\
@> >> H^2(\partial N(X_{5,0});A) \\
@> >> H^3(N(X_{5,0}),\partial N(X_{5,0});A) \cong H_{5}(N(X_{5,0}),A) \cong A \\
@> >> H^3(N(X_{5,0});A) \cong \{0\} @> >>
\end{CD}$$

\ \\

\noindent and we have isomorphims $H^2(\partial N(X_{5,0});A) \cong H_2(\partial N(X_{5,0});A) \cong A \oplus A \oplus A \oplus A$ and $H^3(\partial N(X_{5,0});A) \cong H_3(\partial N(X_{5,0});A) \cong \{0\}$.
Together with Poincar\'e duality theorem and arguments in the original articles, we can know homology groups and cohomology rings of the manifold $\partial N(X_{5,0})$ completely.
We obtain explicit structures of homology and cohomology groups, which were not given explicitly in the original articles. Last, for any coefficient ring $A$ which is a PID, the cohomology ring of the manifold is isomorphic to that of  a manifold represented as a connected sum of four copies of $S^2 \times S^5$.

\begin{Thm}
\label{thm:7}
Let $A$ be a PID and $n=4,5$. Let $\{G_j\}_{j=0}^7$ be a sequence of free and finitely generated modules over $A$ such that $G_0=G_m=A$, that $G_j=G_{7-j}$ for $0 \leq j \leq 7$, that $G_1=G_6=\{0\}$ and that the rank of $G_2$ is greater than or equal to $4$.
In this situation, there exists a pair $(M_1,M_2)$ of $7$-dimensional closed and simply-connected manifolds satisfying the following properties.
\begin{enumerate}
\item $H_j(M_i;A) \cong G_j$ for $i=1,2$ and $0 \leq j \leq 7$.
\item The cohomology rings $H^{\ast}(M_1;A)$ and $H^{\ast}(M_2;A)$ are isomorphic. 
\item $M_1$ admits a special generic map into ${\mathbb{R}}^n$.
\item $M_2$ admits no special generic map into ${\mathbb{R}}^n$.
\item $M_2$ admits a fold map into ${\mathbb{R}}^n$  
\end{enumerate}
This holds for $n=3$ if we add the constraint $G_3=\{0\}$.
Moreover, 
\end{Thm}
\begin{proof}
We set $M_1$ as a suitable manifold admitting a special generic map in Example \ref{ex:1}.
We set $M_2$ as a manifold represented as a connected sum of $\partial N(X_{5,0})$, which is discussed before the present theorem, and a suitable manifold admitting a special generic map in Example \ref{ex:1}.
$M_2$ admits a fold map into ${\mathbb{R}}^n$ since we can embed $M_2$ in ${\mathbb{R}}^8$ together with classical theory of existence of fold maps on manifolds in \cite{eliashberg} and \cite{eliashberg2}. 
This together with Theorem \ref{thm:6} completes the proof.
\end{proof}

Last, we can also give a remark similar to Remark \ref{rem:2} on Theorem \ref{thm:5} by using the manifold $\partial N(X_{5,0})$ and we leave this to readers.
\section{Acknowledgement}
\thanks{The author is a member of the following project: JSPS KAKENHI Grant Number JP17H06128 "Innovative research of geometric topology and singularities of differentiable mappings"
(https://kaken.nii.ac.jp/en/grant/KAKENHI-PROJECT-17H06128/). This work is supported by

 "The Sasakawa Scientific Research Grant" (2020-2002 : https://www.jss.or.jp/ikusei/sasakawa/)..}


\begin{thebibliography}{30}
\bibitem{dranishnikovrudyak} A. N. Dranishnikov and Y. B. Rudyak, \textsl{Examples of non-formal closed simply connected manifolds of dimensions 7 and more}, arXiv:math/0306299v3.
\bibitem{dranishnikovrudyak2} A. N. Dranishnikov and Y. B. Rudyak, \textsl{Examples of non-formal closed {\rm (}$k-1${\rm )}-connected manifolds of dimensions $4k-1$ and more}, arXiv:math/0306299.
\bibitem{eellskuiper} J. J. Eells and N. H. Kuiper, \textsl{An invariant for certain smooth manifolds}, Ann. Mat. Pura Appl. 60 (1962), 93--110.
\bibitem{eliashberg} Y. Eliashberg, \textsl{On singularities of folding type}, Math. USSR Izv. 4 (1970). 1119--1134.
\bibitem{eliashberg2} Y. Eliashberg, \textsl{Surgery of singularities of smooth mappings}, Math. USSR Izv. 6 (1972). 1302--1326.
\bibitem{fernandezMunoz} M. Fern\'andez and V. mu\~noz, \textsl{On non-formal simply connected manifolds}, Topology Appl. 135 Issues 1--3 (2004), 111--117. math.DG/0212141.
\bibitem{kitazawa} N. Kitazawa, \textsl{On round fold maps} (in Japanese), RIMS Kokyuroku Bessatsu B38 (2013), 45--59.
\bibitem{kitazawa2} N. Kitazawa, \textsl{On manifolds admitting fold maps with singular value sets of concentric spheres}, Doctoral Dissertation, Tokyo Institute of Technology (2014).
\bibitem{kitazawa3} N. Kitazawa, \textsl{Fold maps with singular value sets of concentric spheres}, Hokkaido Mathematical Journal Vol.43, No.3 (2014), 327--359.
\bibitem{kitazawa4} N. Kitazawa, \textsl{Notes on fold maps obtained by surgery operations and algebraic information of their Reeb spaces}, submitted to a refereed journal, arxiv:1811.04080.
\bibitem{kitazawa5} N. Kitazawa \textsl{Notes on explicit smooth maps on 7-dimensional manifolds into the 4-dimensional Euclidean space}, submitted to a refereed journal, arxiv:1911.11274.
\bibitem{kobayashisaeki} M. Kobayashi and O. Saeki, \textsl{Simplifying stable mappings into the plane from a global viewpoint}, Trans. Amer. Math. Soc. 348 (1996), 2607--2636. 
\bibitem{kraines} D. Kraines, \textsl{Massey higher products}, Trans. Amer. Math. Soc. 124 (1966), 431--449. 
\bibitem{kreck} M. Kreck,  \textsl{On the classification of $1$-connected $7$-manifolds with torsion free second homology}, to appear in the Journal of Topology, arxiv:1805.02391.
\bibitem{masseyuehara} W. Massey and H. Uehara, \textsl{the Jacobi identity for Whitehead products}, Algebraic geometry and topology. A symposium in honor of S. Lefschetz, pp. 361--377, Princeton University Press, Princeton, N. J., 1957.
\bibitem{milnor} J. Milnor, \textsl{On manifolds homeomorphic to the $7$-sphere}, Ann. of Math. (2) 64 (1956), 399--405.
\bibitem{reeb} G. Reeb, \textsl{Sur les points singuliers d\`{u}ne forme de Pfaff completement integrable ou d’une fonction numerique}, -C. R. A. S. Paris 222 (1946), 847--849. 
\bibitem{saeki} O. Saeki, \textsl{Topology of special generic maps of manifolds into Euclidean spaces}, Topology Appl. 49 (1993), 265--293.
\bibitem{saeki2} O. Saeki, \textsl{Topology of special generic maps into $\mathbb{R}^3$}, Workshop on Real and Complex Singularities (Sao Carlos, 1992), Mat. Contemp. 5 (1993), 161--186.
\bibitem{saekisakuma} O. Saeki and K. Sakuma, \textsl{On special generic maps into ${\mathbb{R}}^3$}, Pacific J. Math. 184 (1998), 175--193.
\bibitem{saekisakuma2} O. Saeki and K. Sakuma, \textsl{Special generic maps of $4$-manifolds and compact complex analytic surfaces}, Math. Ann. 313, 617--633, 1999.
\bibitem{saekisuzuoka} O. Saeki and K. Suzuoka, \textsl{Generic smooth maps with sphere fibers} J. Math. Soc. Japan Volume 57, Number 3 (2005), 881--902.
\bibitem{shiota} M. Shiota, \textsl{Thom's conjecture on triangulations of maps}, Topology 39 (2000), 383--399.
\bibitem{taylor} L. R. Taylor, \textsl{Controlling indeterminacy in Massey triple products}, Geom. Dedicata 148 (2010), 371--389. 
\bibitem{taylor2} L. R. Taylor, \textsl{Massey Triple Products}, https://www3.nd.edu/\~{}taylor/talks/2011-03-22-Princeton.pdf, Princeton Topology Seminar, 2011/3/22. 
\bibitem{thom} R. Thom, \textsl{Les singularites des applications differentiables}, Ann. Inst. Fourier (Grenoble) 6 (1955-56), 43--87.
\bibitem{whitney} H. Whitney, \textsl{On singularities of mappings of Euclidean spaces: I, mappings of the plane into the plane}, Ann. of Math. 62 (1955), 374--410.
\bibitem{wrazidlo} D. J. Wrazidlo, \textsl{Standard special generic maps of homotopy spheres into Eucidean spaces}, Topology Appl. 234 (2018), 348--358, arxiv:1707.08646.
\end{thebibliography}
\end{document}